\documentclass[12pt]{article}

\usepackage[usenames]{color}
\usepackage{amssymb}
\usepackage{amsmath}
\usepackage{amsthm}
\usepackage{amsfonts}
\usepackage{amscd}
\usepackage{graphicx}

\usepackage[colorlinks=true,
linkcolor=webgreen,
filecolor=webbrown,
citecolor=webgreen]{hyperref}

\definecolor{webgreen}{rgb}{0,.5,0}
\definecolor{webbrown}{rgb}{.6,0,0}

\usepackage{color}
\usepackage{fullpage}
\usepackage{float}

\usepackage{latexsym}
\usepackage{epsf}
\usepackage{breakurl}

\newcommand{\seqnum}[1]{\href{https://oeis.org/#1}{\rm \underline{#1}}}

\begin{document}


\theoremstyle{plain}
\newtheorem{theorem}{Theorem}
\newtheorem{corollary}[theorem]{Corollary}
\newtheorem{lemma}[theorem]{Lemma}
\newtheorem{proposition}[theorem]{Proposition}

\theoremstyle{definition}
\newtheorem{definition}[theorem]{Definition}
\newtheorem{example}[theorem]{Example}
\newtheorem{conjecture}[theorem]{Conjecture}

\theoremstyle{remark}
\newtheorem{remark}[theorem]{Remark}

\begin{center}
\vskip 1cm{\Large\bf
On Hofstadter's G-sequence
}
\vskip 1cm
F. M. Dekking \\
CWI Amsterdam and Delft University of Technology \\
Faculty EEMCS \\
P.O.~Box 5031\\
2600 GA Delft \\
The Netherlands\\
\href{mailto:F.M.Dekking@math.tudelft.nl}{\tt F.M.Dekking@math.tudelft.nl} \\
\end{center}

\vskip .2 in

\begin{abstract}
We characterize the entries of Hofstadter's G-sequence in terms of the lower and upper Wythoff sequences. This can be used to give a short and comprehensive proof of the equality of Hofstadter's G-sequence and  the sequence of averages of the swapped Wythoff sequences. In a second part we give some results that hold when one replaces the golden mean by other quadratic algebraic numbers. In a third part we prove  a close relationship between Hofstadter's G-sequence and a sequence studied by Avdivpahi\'c and Zejnulahi.
\end{abstract}

\section{Introduction}

Hofstadter's G-sequence $G$  is defined by $G(1)=1, G(n)=n-G(G(n-1)))$ for $n\ge 2$.

\begin{table}[H]
\begin{center}
\begin{tabular}{c|ccccccccccccccccccc}
$n$ &                  0 & 1 & 2 & 3 & 4 & 5 & 6 & 7 & 8 & 9 & 10 & 11 & 12 & 13 & 14 & 15 & 16 & 17 & 18 \\
\hline $G(n)$ & 0 &  1 &  1 &  2 &  3 &  3 &  4 &  4 &  5 &  6 &  6 &  7 &  8 &  8 &  9 &  9 &  10 &  11 &  11\\[-.3cm]
\end{tabular} \end{center}
\caption{$G(n)$ = A005206($n$) for $n=0,\ldots,18$.} \label{Hof}
\end{table}

 It was proved in 1988, independently in the two articles \cite{Gault,Granville} that there is a simple expression for Hofstadter's G-sequence as a slow Beatty sequence, given for $n\ge 0$ by

\begin{equation}\label{eq:GB}
  G(n) = \lfloor (n+1)\gamma \rfloor,
\end{equation}
where $\gamma=(\sqrt{5}-1)/2$, the small golden mean.

The terminology `slow Beatty sequence' comes from the paper \cite{kimbstol} by Kimberling and Stolarsky.
Actually, $G$ is \emph{not} a Beatty sequence: Beatty sequences are  sequences $( \lfloor n\alpha  \rfloor))$ with $\alpha$ a positive real number \emph{larger} than 1. See, e.g., the papers \cite{Beatty, Stolarski}.

The paper by Kimberling and Stolarsky provides the following basic result.

\begin{theorem}{\bf [Kimberling and Stolarsky]} \label{th:KS}
Suppose that $\gamma$ in $(0, 1)$ is irrational, and let $ s(n) = \lfloor (n+1)\gamma \rfloor$ for $n\ge 0$. Let $A$ be the set $\{n\ge 0 : s(n+1) = s(n)\}$ and let $a(0) < a(1) < \ldots$ be the members of $A$ in increasing order.  Similarly, let $b$ be the sequence of those n such
that $s(n + 1) = s(n) + 1$. Then $a$ is the Beatty sequence of $1/(1-\gamma)$, and  $b$ is the Beatty sequence of $1/\gamma$.
\end{theorem}

When we apply this result to determine the value of $s(n)$ for a given $n$, then we need information on \emph{two} entries, namely $s(n)$ and $s(n+1)$, and given this information we do not yet know for which $m$ $s(n)$ will be equal to $a(m)$, respectively $b(m)$.
The following theorem is more useful in this respect.

\begin{theorem} \label{th:GUL} Suppose that $\gamma$ in $(0, 1)$ is irrational, and let $ s(n+1) = \lfloor n\gamma \rfloor$ for $n\ge 0$.

Let $L(n)=\big\lfloor\frac1{\gamma}n\big\rfloor$, and $U(n)=\big\lfloor\frac1{1-\gamma}n\big\rfloor$ for $n\ge 0$.
Then for all $n\ge 0$\\[-.3cm]
$$s(L(n))=n, \quad s(U(n))=\frac{\gamma}{1-\gamma}n.$$
\end{theorem}

\begin{proof} By definition, the sequence $L$ satisfies for $n\ge 0$
  \begin{equation*}\label{eq:LWS}
 \frac1{\gamma}n = L(n)+\varepsilon_n
  \end{equation*}
for some $\varepsilon_n$ with $0\le \varepsilon_n<1$.   This leads to
 \begin{equation*}
s(L(n))= \big\lfloor (L(n)+1)\gamma\big\rfloor = \big\lfloor n +\gamma(1-\varepsilon_n)\big\rfloor=n,
\end{equation*}
since obviously  $0\le\gamma(1-\varepsilon_n)<1$.

By definition, the sequence $U$ satisfies for $n\ge 0$
  \begin{equation*}\label{eq:WS}
  \frac1{1-\gamma}n = U(n)+\varepsilon_n',
  \end{equation*}
for some $\varepsilon_n'$ with $0\le \varepsilon_n'<1$.   This leads to
 \begin{equation*}
s(U(n))= \Big\lfloor (U(n)+1)\gamma\Big\rfloor =  \Big\lfloor\frac{\gamma}{1-\gamma}n+\gamma(1-\varepsilon_n')\Big\rfloor= \frac{\gamma}{1-\gamma}n,
\end{equation*}
since  obviously $0\le\gamma(1-\varepsilon_n')<1$.
\end{proof}

It is well-known (see \cite{Beatty}) that if $\alpha$ and $\beta$ are two real numbers larger than 1, and moreover $1/\alpha+1/\beta = 1$, then $\alpha$ and $\beta$ form a \emph{complementary Beatty pair}, which means that the two sets $\{\lfloor n\alpha  \rfloor, n\ge 1\}$ and  $\{\lfloor n\beta  \rfloor, n\ge 1\}$ are disjoint, and that their union contains every positive integer.
Note that for all $\gamma$ in $(0, 1)$ the sequences $L$ and $U$ form a complementary Beatty pair, since $\frac1{\gamma}>1, \frac1{1-\gamma}>1$ and $(\frac1{\gamma})^{-1}+(\frac1{1-\gamma})^{-1}=1$.

\section{Hofstadter and Wythoff}

The most famous complementary Beatty pair is obtained by choosing $\alpha=\varphi$, and $\beta=\varphi^2$, where $\varphi:=(1+\sqrt{5})/2$ is the golden mean. The Beatty sequences  $L(n)=\lfloor n \varphi \rfloor$ and $U(n)=\lfloor n \varphi^2 \rfloor$ for $n\ge 1$ are known as the \emph{lower Wythoff sequence} and the \emph{upper Wythoff sequence}. The name giving has its origins in the paper \cite{Wyt}.

\begin{table}[H]
\begin{center}
\begin{tabular}{c|ccccccccccccccccccc}
$n$ &               1 & 2 & 3 & 4 & 5 & 6  & 7  & 8  & 9  & 10 & 11 & 12 & 13 & 14 & 15 & 16 & 17 & 18  \\
\hline $ L(n)$ &    1 & 3 & 4 & 6 & 8 & 9 & 11  & 12 & 14 & 16 & 17 & 19 & 21 & 22 & 24 & 25 & 27 & 29  \\
\hline $ U(n)$ &    2 & 5& 7  & 10& 13& 15& 18  & 20 & 23 & 26 & 28 & 31 & 34 & 36 & 39 & 41 & 44 & 47 \\[-.5cm]
\end{tabular} \end{center}
\caption{$ L(n)$ = A000201($n$) and $U(n)$= A001950($n$) for $n=0,\ldots,18$.} \label{lowup}
\end{table}

We next turn our attention to sequence \seqnum{A002251}, described as:
 		Start with the nonnegative integers; then swap $L(k)$ and $U(k)$ for all $k \ge 1$, where  $L$ and $U$ are the lower and upper Wythoff sequences.

 This means that this sequence, which we call $W,$ is defined by

\begin{equation}\label{eq:W}
  W(L(n)) = U(n),\;  W(U(n) )= L(n) \text{ for all } n\ge 1.
\end{equation}

Regrettably, the sequence $W$ has been given offset $0$ in OEIS. One of the unpleasant consequences of the useless addition of $0$ is that sequence \seqnum{A073869} is not a clean Cesar\'o average of \seqnum{A002251}. Another unpleasant consequence is that \seqnum{A073869} is basically a copy of \seqnum{A019444}.

\begin{table}[H]
\begin{center}
\begin{tabular}{c|ccccccccccccccccccc}
$n$ &                  0 & 1 & 2 & 3 & 4 & 5 & 6 & 7 & 8 & 9 & 10 & 11 & 12 & 13 & 14 & 15 & 16 & 17 & 18 \\
\hline $ W(n)$ & 0 &  2 & 1 &  5 &  7 &  3 &  10 &  4 &  13 &  15 &  6 &  18 &  20 &  8 &  23 &  9 &  26 &  28 &  11\\[-.5cm]
\end{tabular} \end{center}
\caption{$ W(n)$ = A002251($n$) for $n=0,\ldots,18$.} \label{SLU}
\end{table}	

The sequence $W$ has the remarkable property that the sum of the first $n+1$ terms is divisible by $n+1$.
This leads to the sequence \seqnum{A073869}, defined as  A073869$(n) = \sum_{i=0}^{n} W(i)/(n+1)$.

\begin{table}[H]
\begin{center}
\begin{tabular}{c|ccccccccccccccccccc}
$n$ &               0 & 1 & 2 & 3 & 4 & 5 & 6 & 7 & 8 & 9 & 10 & 11 & 12 & 13 & 14 & 15 & 16 & 17 & 18 \\
\hline\\[-0.4cm]
$\overline{W}(n)$ & 0 &  1 &  1 &  2 &  3 &  3 &  4 &  4 &  5 &  6 &  6 &  7 &  8 &  8 &  9 &  9 &  10 &  11 &  11\\[-.5cm]
\end{tabular} \end{center}
\caption{$\overline{W}(n)$ = A073869($n$) for $n=0,\ldots,18$.} \label{WALUW}
\end{table}

The following theorem is a conjecture by Murthy in \cite[\seqnum{A073869}]{oeis}, but is proved in the long paper \cite{Venk}. We give a new short proof.

\begin{theorem} \label{th:main} The averaged Wythoff swap sequence $\overline{W}$ is equal to Hofstadter's G-sequence.
\end{theorem}

\begin{proof}
The result holds for $n=0,1$. It suffices therefore to consider the sequence of differences.
Subtracting $G(n-1)=\sum_{i=0}^{n-1} W(i)/n$ from $G(n)=\sum_{i=0}^{n} W(i)/(n+1)$, we see that we have to  prove for all $n\ge 2$
\begin{equation}\label{eq:rec}
  (n+1)G(n)-nG(n-1)=W(n).
\end{equation}
But we know that there are only two possibilities for the recursion from $G(n-1)$ to $G(n)$. Therefore Equation \eqref{eq:rec} turns into the following two equations.
\begin{align}
  G(n)=G(n-1)&\; \Rightarrow\; G(n) = W(n),\label{eq:GW} \\
  G(n)= G(n-1)+1&\; \Rightarrow\; G(n) = W(n)-n.\label{eq:GW2}
\end{align}
It is not clear how to prove these equalities directly. However, we can exploit Theorem \ref{th:KS}.
According to this theorem with $s=G$, and $\gamma=(\sqrt{5}-1)/2$, and so $1/\gamma=\varphi$, $1/(1-\gamma)=\varphi^2$,
\begin{align}
  G(n)=G(n-1)&\; \Leftrightarrow \; \exists M \text{ such that } n=U(M),\label{eq:GU} \\
  G(n)= G(n-1)+1 &\; \Leftrightarrow\;   \exists M \text{ such that } n=L(M)\label{eq:GL} .
\end{align}
So we first have to prove that $n=U(M)$ implies $G(n) = W(n)$. This holds indeed by an application of Theorem \ref{th:GUL} and Equation \eqref{eq:W}:
$$G(n)=G(U(M)=L(M)=W(U(M))=W(n).$$
Similarly, for the second case $n=L(M)$:
$$G(n)=G(L(M))=M=U(M)-L(M)=W(L(M))-L(M)=W(n)-n.$$
Here we applied $U(M)=L(M)+M$ for $M\ge 1$, a direct consequence of $\varphi^2M=(\varphi+1)M$.
\end{proof}

In the comments of \seqnum{A073869} there is a scatterplot by Sloane---cf.~Figure \ref{fig:scat}. The points have a nice symmetric distribution around the line $y=x$, since the points consists of all pairs $(L(n),U(n))$ and $(U(n),L(n))$ for $n=1,2,\dots$. (Ignoring (0,0).) Apparently the points are almost lying on two lines. What are the equations of these lines? This is answered by the following proposition.

\begin{figure}[ht]
\vspace*{-0.1 cm}
\begin{center}
\includegraphics[width=8.0cm]{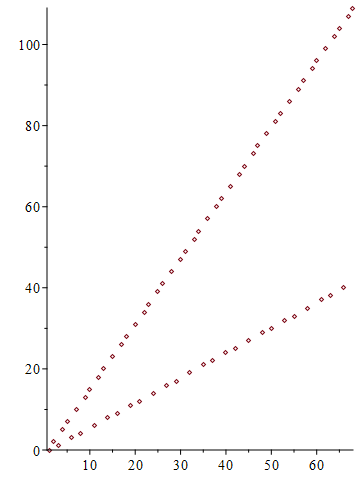}
\caption{\small Scatterplot of the first 68 entries of $W$.}\label{fig:scat}
\vspace*{-0.1cm}
\end{center}
\end{figure}

\begin{proposition} Let $W$ be the Wythoff swap sequence, and $\gamma=1/\varphi$. Then for all $n\ge 1$
$$W(U(n))=\lfloor \gamma U(n)\rfloor,  W(L(n))=\lfloor \varphi L(n) \rfloor +1.$$
\end{proposition}

\begin{proof}
Equation \eqref{eq:GW} and Equation \eqref{eq:GW2} yield
\begin{align*}
  W(n)& = G(n),\;  & \text{ if }  G(n)=G(n-1),\\
  W(n)& = G(n)+n & \; \text{ if }  G(n)= G(n-1)+1.
\end{align*}
Since $G(n)=\lfloor (n+1)\gamma \rfloor$ by Equation \eqref{eq:GB}, it follows from Equation \eqref{eq:GU} that
$$W(U(M))=\lfloor U(M)\gamma \rfloor.$$
Since all $M=1,2\ldots$ will occur, this gives the first half of the proposition.

For the second half of the proposition we perform the following computation under the assumption that $n=L(M)$:
$$ G(n)+n = G(n-1) + n +1 = \lfloor n\gamma \rfloor +n +1 = \lfloor n(\gamma+1) \rfloor +1 = \lfloor n\varphi \rfloor +1. $$
Now Equation \eqref{eq:GL} gives that $W(L(M))=\lfloor \varphi L(M) \rfloor +1$.
\end{proof}

\begin{remark} Simple applications of Theorem \ref{th:main} prove the  conjectures in \seqnum{A090908} (Terms $a(k)$ of A073869 for which $a(k)=a(k+1)$.), and \seqnum{A090909} (Terms $a(k)$ of A073869 for which $a(k-1), a(k)$ and $a(k+1)$ are distinct.). It also proves the conjectured values of sequence \seqnum{A293688}.
\end{remark}

\section{Generalizations}

There is a lot of literature on generalizations of Hofstadter's recursion $G(n)=n-G(n-1))$. In most cases there is no simple description of the sequences that are generated by such recursions.
An exception is the recursion $V(n)=V(n-V(n-1))+V(n-V(n-4))$ analysed by Balamohan et al.\ \cite{Balamohan}.
The sequence with initial values 1,1,1,1 generated by this recursion is sequence \seqnum{A063882}.
Allouche and Shallit \cite{AllSha} prove  that the `frequencies' of this sequence can be generated by an automaton.
See the recent paper \cite{Fox} for more results on this type of Hofstadter's recursions, known as Hofstadter Q-sequences.	
We consider the  paper \cite{Celaya}, that  gives a direct generalization of Hofstadter's G-sequence.

\begin{theorem}{\bf [Celaya and Ruskey]}\label{th:CR}	Let $k\ge 1$, and let $\gamma = [0; k, k, k, \ldots]$.
Assume $H(n) = 0$ for $n < k$, and for $n \ge k$, let\\[-.5cm]
$$H(n) = n-k+1 - \Big(\sum_{i=1}^{k-1} H(n-i)\Big)- H(H(n-k).$$
Then for $n\ge 1$, $H(n) = \lfloor \gamma(n+1)\rfloor$.
\end{theorem}

As an example, we take the case $k=2$. In that case $\gamma=\sqrt{2}-1$, the small silver mean. The recursion for what we call the Hofstadter Pell sequence is
$$H(n)=n-1-H(n-1)-H(H(n-2)).$$
Here Theorem  \ref{th:CR} gives that
$$(H(n))=\lfloor \gamma(n+1)\rfloor=0, 0, 1, 1, 2, 2, 2, 3, 3, 4, 4, 4, 5, 5, 6, 6, 7, 7, 7, 8, 8, 9, 9, 9,10,10,\ldots.$$
This is sequence \seqnum{A097508} in OEIS.

Let $1/\gamma=1+\sqrt{2}$ and $1/(1-\gamma)=1+\frac12\sqrt{2}$ form the Beatty pair given by Theorem \ref{th:KS}.
Let $L^{\rm P}=(\lfloor n(1+\sqrt{2})\rfloor)$ and $U^{\rm P}=(\lfloor n(1+\frac12\sqrt{2})\rfloor)$ be the associated Beatty sequences.
One has $L^{\rm P}=\seqnum{A003151}$, and $U^{\rm P}=\seqnum{A003152}$.

According to Theorem \ref{th:GUL}, with $R$ the slow Beatty sequence \seqnum{A049472} given by $R(n)=\lfloor\frac12\sqrt{2}n\rfloor$, the following holds for the Hofstadter  Pell sequence $H$:
$$H(L^{\rm P}(n))=n,\;H(U^{\rm P}(n))=R(n),\text{ for all } n\ge 1.$$

The sequence with  $L^{\rm P}$ and $U^{\rm P}$ swapped is

\seqnum{A109250} $=2, 1, 4, 3, 7, 9, 5, 12, 6, 14, 16, 8, 19, 10, 21, 11, 24, 26\dots$.\\
Apparently there is nothing comparable to the averaging phenomenon that occurred in the golden mean case.

\begin{remark} See \seqnum{A078474}, and in particular \seqnum{A286389}  for two generalizations of Hofstadter's recursion, with conjectured expressions similar to Equation \eqref{eq:GB}. The conjecture for \seqnum{A286389} is recently proved in Shallit \cite{Remark7}.

For the recursion $a(n)=n-\lfloor \frac12 a(a(n-1))\rfloor$ given in \seqnum{A138466} it is proved by Cloitre that   $(a(n))$ satisfies  Equation \eqref{eq:GB} with $\gamma=\sqrt{3}-1$. For generalizations of this see \seqnum{A138467}.
\end{remark}

\section{Greediness}
There is a more natural way to define the Wythoff swap sequence $W$, which at first sight has nothing to do with Wythoff sequences.
 Venkatachala in the paper \cite{Venk} considers the following greedy algorithm:\\
$f(1) = 1$,  and for $n\ge 2, f(n)$ is the least natural number such that
$$(a) \; f(n) \not\in \{ f(1),\ldots,f(n-1)\};\quad (b) \;   f(1)+f(2)+\cdots+ f(n) \;\text{ is divisible by }\; n.$$
Surprisingly, it follows from Venkatachala's analysis that one has for all $n\ge 1$
$$W(n)=f(n+1)-1.$$
 The recent paper \cite{Avdip} studied a sequence $z$ defined by a similar greedy algorithm:\\
 $z(1) = 1$,  and for $n\ge 2, z(n)$ is the least natural number such that
 $$(a) \; z(n) \not\in \{ z(1),\ldots,z(n-1)\};\quad (b) \;   z(1)+z(2)+\cdots+ z(n) \equiv 1 \mod(n+1).$$
 This entails that $(m(n))$, defined by $m(n):=(z(2)+\ldots+z(n))/(n+1)$ for $n\ge 1$, is a sequence of integers.

 These sequences have been analysed by Shallit in the paper \cite{Walnut} using the computer software Walnut.
Our Theorem \ref{th:AZ} is an improvement of  \cite[Theorem 6]{Walnut}. In the proof of Theorem \ref{th:AZ} we need  the values of the Wythoff sequences at the Fibonacci numbers.

\begin{lemma} \label{lem:LUvalues}  Let $L$ and $U$ be the Wythoff sequences. Then for all $k\ge 1$
\begin{align}
& L(F_{2k}) =F_{2k+1}-1;\label{eq:LF2k}\\
& U(F_{2k}) =F_{2k+2}-1;\label{eq:UF2k}\\
& L(F_{2k-1})=F_{2k};\label{eq:x}\\
& U(F_{2k-1}) =F_{2k+1}.\label{eq:UFodd}
\end{align}
\end{lemma}

\begin{proof}
These equations can be derived from  \cite{Avdip}[Lemma 2.D]. Another, easy, proof is based on recalling that $L(m)$ gives the position of the $m^{\rm th}$ $0$ in the infinite Fibonacci word $0100101\dots$ generated by the morphism $\mu: 0\mapsto 01, 1\mapsto 0$ (see, e.g., \cite[Corollary 9.1.6]{AllShall}). The infinite Fibonacci word is the limit of the words $\mu(0)=01, \mu^2(0)=010, \mu^3(0)=01001, \mu^4(0)=01001010,\ldots$.

Let $|w|$, $|w|_0$, and $|w|_1$  denote the length, the number of 0's and the number of 1's of a word $w$. Then it is easy to see that
\begin{equation}\label{eq:mun}
|\mu^m(0)|=F_m,\quad  |\mu^m(0)|_0=F_{m-1},\quad |\mu^m(0)|_1=F_{m-2},
\end{equation}
for all $m\ge 1$, where $\mu^m$ is the $m^{\rm th}$ iterate of $\mu$.
Since $\mu^m(0)$ ends in $01$ for odd $m$, Equation (\ref{eq:mun}) with $m=2k+1$  implies that Equation (\ref{eq:LF2k}) holds.
Similarly, since $\mu^m(0)$ ends in $10$ for even $m$, one obtains Equation (\ref{eq:UF2k}).
That Equation (\ref{eq:x}) is correct follows from Equation (\ref{eq:mun}) with $m=2k$ , since $\mu^m(0)$ ends with $0$ for odd $m$.
Similarly, since $\mu^m(0)$ ends with $1$ for odd $m$, one obtains Equation (\ref{eq:UFodd}) with $m=2k+1$.
\end{proof}

\begin{theorem} \label{th:AZ}  Let $(z(n))$ and $(m(n))$ be the  Avdivpahic and Zejnulahi sequences.
Let $W$ be the Wythoff swap sequence. Then for all $n\ge 1$
\begin{align*}
\text{\rm (a) } z(n) & =W(n)\quad \text{\rm except if  } n=F_{2k+1}\!-\!1 \text{ \rm or } n= F_{2k+1},\\
 \text{\rm In fact, }  &  z(F_{2k+1}\!-\!1)=F_{2k}\!-\!1, z(F_{2k+1})=F_{2k+2},\\
                       &  W(F_{2k+1}\!-\!1)=F_{2k+2}\!-\!1, W(F_{2k+1})=F_{2k}.  \\
\text{\rm (b) } m(n) & =\overline{W}(n)\quad \text{\rm except if } n= F_{2k+1}\!-\!1.\\
 \text{\rm In fact, }  & m(F_{2k+1}\!-\!1)=F_{2k}\!-\!1, \overline{W}(F_{2k+1}\!-\!1)=F_{2k}.
\end{align*}
\end{theorem}
\begin{proof}
Part (a). We use the formula for $z(n)$ proved in the paper \cite{Avdip}:
\begin{displaymath}
z(n) = \begin{cases}
F_{k-1}-1 & \text{if } n=F_{k}-1; \\
F_{k+1}   & \text{if } n=F_k; \\
L(k)      & \text{if } n=U(k); \\
U(k)      & \text{if } n=L(k).
\end{cases}
\end{displaymath}
So $z$ is the swapping of $L$ and $U$ for indices $n\neq F_k-1$ and $n\neq F_k$.
We first handle the case of the Fibonacci numbers with an odd index. Here we have to prove that
\begin{align}
W(F_{2k+1}-1) & =F_{2k+2}-1 \label{eq:WFodd}\\
W(F_{2k+1})  &  =F_{2k} \label{eq:WFeven}.
\end{align}
We start with Equation (\ref{eq:WFodd}).
We have to show that  there exists $m$ such that either the pair of equations  $L(m)=F_{2k+1}-1$, and $U(m)=F_{2k+2}-1$, or the pair of equations $U(m)=F_{2k+1}-1$, and $L(m)=F_{2k+2}-1$ holds.
The first pair of these swapping equations, with the value $m=F_{2k}$, is equal to the equations (\ref{eq:LF2k}), and (\ref{eq:UF2k}), as we see from  Lemma \ref{lem:LUvalues}.

Next, we prove Equation (\ref{eq:WFeven}).
Here Lemma \ref{lem:LUvalues} gives that Equation (\ref{eq:UFodd}) and Equation (\ref{eq:x}) solve the swapping equations.

\medskip

We still have to handle the case with Fibonacci numbers with an even index. There we have to prove that
\begin{align}
W(F_{2k}-1) & =z(F_{2k}-1)=F_{2k-1}-1 \label{eq:WFzodd}\\
 W(F_{2k})  & =z(F_{2k})=F_{2k+1}\label{eq:WFzeven}.
\end{align}
We start with Equation (\ref{eq:WFzeven}).
Here Lemma \ref{lem:LUvalues} gives that  Equation(\ref{eq:x}) and Equation(\ref{eq:UFodd})  solve the swapping equations.

Next, we prove Equation (\ref{eq:WFzodd}). Here Lemma \ref{lem:LUvalues} gives that
  Equation (\ref{eq:LF2k}) and Equation(\ref{eq:UF2k})  solve the swapping equations, both with $k$ shifted by 1.

\bigskip

Part (b). The case $n=1$: for $k=1,\;F_3-1=1$, and $m(1)=1=\overline{W}(1)-1$.

For $n\ge 2$ we have
$$(n+1)m(n)=z(2)+\cdots+z(n), \;(n+1)\overline{W}(n)= W(1)+\cdots+W(n).$$
We see that for $n=2, \; 3m(2)=z(2)=3$, and $3\overline{W}(n)=W(1)+W(2)=2+1=3$.\\
So also
$$(n+1)m(n)=3+z(3)+\cdots+z(n), \;(n+1)\overline{W}(n)=3+ W(3)+\cdots+W(n).$$
Note furthermore that $z(F_{2k+1}-1)+z(F_{2k+1})=F_{2k}-1+F_{2k+2},$ and $\overline{W}(F_{2k+1}-1)+\overline{W}(F_{2k+1})=F_{2k+2}-1+F_{2k}.$\\
Since these two sums are equal, the difference of 1 created at $n=F_{2k+1}-1$ is `repaired' at  $n=F_{2k+1}$. This proves the first part of Part b).

 For the second part we have to see that $m(F_{2k+1}\!-\!1)=F_{2k}-1$, or equivalently (see the proof of the first part of Part (b)), that $\overline{W}(F_{2k+1}-1)=F_{2k}$.

In general we have for all $n\ge 1$ (see Theorem \ref{th:main}):
$$\overline{W}(n)=G(n)=\lfloor (n+1)\gamma\rfloor=\lfloor (n+1)/\varphi\rfloor=\lfloor (n+1)\varphi\rfloor-(n+1)=L(n+1)-(n+1).$$
So we obtain, using Equation(\ref{eq:x})
\begin{equation}\label{eq:W2kplus1}
\overline{W}(F_{2k+1}-1)=L(F_{2k+1})-F_{2k+1}=F_{2k+2}-F_{2k+1}=F_{2k},
\end{equation}
which ends the proof of Part (b).
\end{proof}

 Let $(a(n),b(n))$ defined by the recurrences  $a(n))=n-b(a(n-1)), b(n)=n-a(b(n-1))$ be the ``married" functions of Hofstadter given in his book \cite[p.\ 137]{Hofst}. Here $(a(n))$ is \seqnum{A005378} and $(b(n))$ is \seqnum{A005379}.
\begin{theorem}{\bf [Stoll]\cite{Stoll}} \label{th:Stoll}  Let $(a(n))$ and $(b(n))$ be the ``married"  sequences of Hofstadter. Let $\gamma=(\sqrt{5}-1)/2$ be  the small golden mean.  Then for all $n\ge 1$\\[-.8cm]
\begin{align*}
\text{\rm (a) } a(n) & =\lfloor (n+1)\gamma \rfloor \quad \text{\rm except if } n=F_{2k}\!-\!1:  a(F_{2k}\!-\!1)=\lfloor F_{2k}\gamma\rfloor +1.\\
\text{\rm (b) } b(n) & =\lfloor (n+1)\gamma \rfloor \quad \text{\rm except if } n=F_{2k+1}\!-\!1:  b(F_{2k+1}\!-\!1)=\lfloor F_{2k+1}\gamma\rfloor - 1.
\end{align*}
\end{theorem}

It follows by combining Theorem \ref{th:main}, Theorem \ref{th:AZ}, Stoll's Theorem \ref{th:Stoll}, and Equation (\ref{eq:W2kplus1}) that $$(b(n))=(m(n)).$$
Then Stoll's theorem also gives an expression for $(a(n))$. See Shallit's paper \cite{Walnut} for proofs using the computer software Walnut.

\section{Acknowledgment} I  thank  Jean-Paul Allouche for useful remarks. Thanks are also due to one referee for pointing out an important reference, and to a second referee for many remarks that have resulted in an improvement of  the presentation.

\bibliographystyle{plain}

\bigskip

\noindent 2020 {\it Mathematics Subject Classification}: Primary 05A17, Secondary 68R15

\noindent \emph{Keywords:} Hofstadter's G-sequence, slow Beatty sequence, Wythoff sequence.

\end{document}